\documentclass[12pt]{amsart}
\usepackage{latexsym, amsmath, amssymb, amsthm, mathrsfs}
\usepackage[alwaysadjust]{paralist}
\usepackage{caption}
\usepackage{subcaption}
\usepackage[T1]{fontenc}
\usepackage{lmodern}
\usepackage[backref,colorlinks=true,linkcolor=blue,urlcolor=blue, citecolor=blue]%
  {hyperref}
\usepackage[shortalphabetic]{amsrefs}
\usepackage[all]{xy}
\usepackage[T1]{fontenc}
\usepackage{mathptmx}
\usepackage{microtype}
\usepackage{framed}
\usepackage{tikz}
\usepackage{graphicx}

\usepackage[centering, includeheadfoot, hmargin=1in, vmargin=1in,
  headheight=30.4pt]{geometry}
\newtheorem{lemma}{Lemma}[section]
\newtheorem{theorem}[lemma]{Theorem}
\newtheorem{corollary}[lemma]{Corollary}
\newtheorem{proposition}[lemma]{Proposition}
\theoremstyle{definition}
\newtheorem{definition}[lemma]{Definition}
\newtheorem{remark}[lemma]{Remark}
\newtheorem{example}[lemma]{Example}

\newtheorem*{theoremA}{Theorem A}
\newtheorem*{theoremB}{Theorem B}
\newtheorem*{theoremC}{Theorem C}

\newtheorem*{corollaryD}{Corollary D}

\renewcommand{\theequation}%
{\arabic{section}.\arabic{lemma}.\arabic{equation}}
\newcommand{\NN}{\ensuremath{\mathbb{N}}} 
\newcommand{\PP}{\ensuremath{\mathbb{P}}} 
\newcommand{\QQ}{\ensuremath{\mathbb{Q}}} 
\newcommand{\RR}{\ensuremath{\mathbb{R}}} 
 
\newcommand{\sI}{\ensuremath{\kern -1pt \mathscr{I}\kern -2pt}} 
\newcommand{\sJ}{\ensuremath{\kern -2pt \mathscr{J}\kern -2pt}} 
\newcommand{\sZ}{\ensuremath{\mathscr{Z}}}
\newcommand{\sO}{\ensuremath{\mathscr{O}}} 
\newcommand{\sU}{\ensuremath{\mathscr{U}}}

\newcommand{\m}{\ensuremath{\mathfrak{m}}}
\renewcommand{\geq}{\geqslant}
\renewcommand{\leq}{\leqslant}

\DeclareMathOperator{\codim}{codim}
\DeclareMathOperator{\mult}{mult}

\DeclareMathOperator{\Supp}{Supp}
\DeclareMathOperator{\ord}{ord}

\DeclareMathOperator{\Zeroes}{Zeroes}
\DeclareMathOperator{\vol}{vol}
\DeclareMathOperator{\intt}{int}
\DeclareMathOperator{\Eff}{Eff}

\newcommand{\equ}{\ensuremath{\,=\,}}
\newcommand{\deq}{\ensuremath{\stackrel{\textrm{def}}{=}}}

\newcommand{\st}[1]{\ensuremath{\left\{ #1 \right\} }}
\DeclareMathOperator{\Bbig}{Big}

\newcommand{\ybul}{\ensuremath{Y_\bullet}}
\newcommand{\dgeq}{\ensuremath{\,\geq\,}}  
\newcommand{\Bplus}{\ensuremath{\textbf{\textup{B}}_{+} }}
\newcommand{\Bminus}{\ensuremath{\textbf{\textup{B}}_{-} }}
\newcommand{\Bstable}{\ensuremath{\textbf{\textup{B}} }}
\newcommand{\dleq}{\ensuremath{\,\leq\,}} 
\newcommand{\origin}{\ensuremath{\textup{\textbf{0}}}}
\newcommand{\ei}{\ensuremath{\textup{\textbf{e}}_i}}
\newcommand{\etwo}{\ensuremath{\textup{\textbf{e}}_2}}
\newcommand{\eone}{\ensuremath{\textup{\textbf{e}}_1}}
\newcommand{\en}{\ensuremath{\textup{\textbf{e}}_n}}
\newcommand{\dsubseteq}{\ensuremath{\,\subseteq\,}}

\newcommand{\HH}[3]{\ensuremath{H^{#1}\left(#2,#3\right)}}

\newcommand{\nob}[2]{\ensuremath{\Delta_{#1}(#2)}}
\newcommand{\inob}[2]{\ensuremath{\widetilde{\Delta}_{#1}(#2)}}
\newcommand{\iss}[1]{\ensuremath{\Delta^{-1}_{#1}}}

\newcommand{\asym}[2]{\ensuremath{\mult_{#1}\| #2 \|}}
\newcommand{\e}{\ensuremath{\epsilon}}

\newcommand{\movs}[2]{\ensuremath{\e(\| #1\|;#2)}}

\begin{document}

\title[Infinitesimal Newton--Okounkov bodies]{Infinitesimal Newton--Okounkov bodies and jet separation}

\author[A.~K\" uronya]{Alex K\" uronya}
\author[V.~Lozovanu]{Victor Lozovanu}

\address{Alex K\"uronya, Johann-Wolfgang-Goethe Universit\"at Frankfurt, Institut f\"ur Mathematik, Robert-Mayer-Stra\ss e 6-10., D-60325
Frankfurt am Main, Germany}
\address{Budapest University of Technology and Economics, Department of Algebra, Egry J\'ozsef u. 1., H-1111 Budapest, Hungary}
\email{{\tt kuronya@math.uni-frankfurt.de}}

\address{Victor Lozovanu, Universit\'a degli Studi di Milano--Bicocca, Dipartimento di Matematica e Applicazioni, 
Via Cozzi 53,, I-20125 Milano, Italy}
\email{{\tt victor.lozovanu@unimib.it}}

\maketitle

\section*{Introduction}
In this paper we wish to continue the investigations initiated  in \cites{KL14,KL15} to find  a satisfactory theory of positivity for divisors in terms of 
convex geometry. To be more specific, our aim here is to relate local positivity of line bundles to Newton--Okounkov bodies attached to infinitesimal flags. 

Ever since the advent of Newton--Okounkov bodies in projective geometry (see \cites{KKh,LM} or the review \cite{Bou1} for an introduction) the main question 
has been how their geometry is connected to the properties of the underlying polarized variety. For example,  attention has been devoted to the combinatorial study of 
Newton--Okounkov bodies in terms of projective data (see for instance \cites{AKL,KLM1,LSS,PSU}). Nevertheless, in order for these invariants to be really useful in the 
quest for  understanding projective varieties, it is more important to uncover implications in the other direction, that is, one should be able to  gain information about
line bundles in terms of their Newton--Okounkov bodies. 

The hope   for such results comes from  Jow's theorem \cite{Jow} claiming that the function associating to an admissible flag the Newton--Okounkov
body of a given divisor determines the latter up to numerical equivalence. Following our earlier work \cites{KL14,KL15}, we are interested in a local version 
of Jow's principle: we will be mostly concerned with the situation where all flags considered are centered at a given point of the variety. 

Compared to \cite{KL15} we specialize the flags further; as suggested by \cite{KL14}*{Sections 3 \& 4}, one can obtain particulary precise results by taking linear flags 
in the exceptional divisor of the blow-up of the point. This way, we can not only achieve a description of ampleness and nefness in terms of infinitesimal Newton--Okounkov bodies, 
but are also able to extend the  convex geometric interpretation of  moving Seshadri constants described in \cite{KL14}  to  all dimensions.

Let now $X$ be a projective variety over the complex numbers, $L$ a big line bundle, and $x\in X$ a closed point. We say that $L$ is locally positive or locally ample at $x$ 
if there exists a neighbourhood $x\in\sU\subseteq X$ such that the Kodaira map $\phi_{mL}$ restricted to $\sU$ is an embedding for all $m\gg 0$. 
One can of course work with the alternative description provided by global generation of large twists of coherent sheaves (cf. \cite{PAGI}*{Example 1.2.21} and 
\cite{Kur10}*{Proposition 2.7}), in any case both conditions end up being equivalent to $x$ belonging to the complement of the augmented base locus $\Bplus(L)$
of $L$ (see \cite{BCL}*{Theorem A}). 

Once a line bundle $L$ has been proven to be locally positive at a point $x\in X$, one can try to measure the extent of its positivity there.  The traditional way to do 
this is via  the Seshadri constant $\e(L;x)$ introduced by Demailly \cite{D} (see also \cite{PAGI}*{Chapter 5} for a thorough introduction and an extensive bibliography), 
or, in our setting,   its extension, the moving Seshadri constant $\movs{L}{x}$ developed by Nakamaye \cite{Nakamaye}, and studied in much more detail by
Ein--Lazarsfeld--Musta\c t\u a--Nakamaye--Popa \cite{ELMNP2}.

Since one can describe both local ampleness and moving Seshadri constants in terms of infinitesimal Newton--Okounkov bodies, the convex-geometric picture of local positivity appears 
to be complete.  The first main result of our  work is a characterization of ampleness and nefness  in terms of Newton--Okounkov bodies 
(cf. \cite{KL14}*{Theorem A} and \cite{KL15}*{Theorems A \& B}, see also \cite{CHPW}).  

To fix terminology, let  $X$ be a smooth projective variety of dimension $n$, $x\in X$ a closed point, and $\pi\colon X'\to X$ be the blow-up of $X$ at $x$  with exceptional divisor $E$. 
An infinitesimal flag $\ybul$ over $x$ is an admissible flag 
\[
\ybul \ : \ Y_0=X' \ \supseteq \ Y_1=E\ \supseteq \ Y_2 \supseteq \ \ldots \ \supseteq \ Y_n \ ,
\]
where each $Y_i$ is a linear subspace of $E\simeq \PP^{n-1}$ of dimension $n-i$ for each $=2,\ldots, n$.  The Newton--Okounkov body of $\pi^*D$ with respect to $\ybul$ on $X'$ 
will be denoted by $\inob{\ybul}{D}$. For further results  regarding infinitesimal Newton--Okounkov bodies the reader is kindly referred to Section 2. 

\begin{theoremA}(Corollary~\ref{cor:nefness})
Let $X$ be a smooth projective variety of dimension $n$,  $D$ a big $\RR$-divisor on $X$. Then the following are equivalent.
\begin{enumerate}
 \item $D$ is nef.
 \item For every point $x\in X$ there exists an infinitesimal flag $\ybul$ over $x$ such that $\origin\in\inob{\ybul}{D}$.
 \item One has  $\origin\in\inob{\ybul}{D}$ for every infinitesimal flag over $X$. 
\end{enumerate}
\end{theoremA}

Before we proceed, let us define what we call the inverted standard simplex of size $\xi>0$: this is the convex body 
\[
 \iss{\xi} \deq \text{convex hull of } \st{\origin, \xi\eone,\xi(\eone+\etwo),\dots,\xi(\eone+\en)} \dsubseteq \RR^n \ ,
\]
where $\eone,\dots,\en$ denote the standard basis vectors for $\RR^n$. Lemma~\ref{lem:2inf} and Proposition~\ref{prop:inverted} below explain how the polytopes 
$\iss{\xi}$ arise very naturally in the infinitesimal setting.

\begin{theoremB}(Corollary~\ref{cor:ampleness}) Let $X$ be a smooth projective variety of dimension $n$, $D$ a big $\RR$-divisor on $X$. Then the following are equivalent. 
\begin{enumerate}
 \item $D$ is ample.
 \item For every point $x\in X$ there exists an infinitesimal flag $\ybul$ over $x$ and a real number $\xi>0$ for which $\iss{\xi}\subseteq \inob{\ybul}{D}$.
 \item $\inob{\ybul}{D}$ contains a non-trivial inverted standard simplex for every infinitesimal flag $\ybul$ over $X$. 
\end{enumerate}
\end{theoremB}

Note that as opposed to \cite{KL15}*{Theorem B}, the theorem above  provides a full generalization of what happens in the surface case; its proof is 
significantly more difficult than that of any of its predecessors. 

An interesting feature of the argument leading to Theorem~B is that it passes through 
separation of jets. In fact, an important step in the proof is Proposition~\ref{prop:jet separation} which claims that line bundles whose infinitesimal Newton--Okounkov
bodies contain large inverted standard simplices will separate many jets. Not surprisingly, we will make an extensive use of the circle of ideas around 
jet separation and moving Seshadri constants, and with it, the non-trivial results of 
\cite{ELMNP2}. Another important  ingredient of the proof is an  acute observation of 
Fulger--Koll\'ar--Lehmann \cite{FKL}*{Theorem~A} linking inequalities between volumes of divisors  to augmented base loci.

It follows from our argument that infinitesimal Newton--Okounkov bodies on projective varieties always contain inverted standard simplices at points where 
the divisor is locally ample. Given an infinitesimal flag $\ybul$, the  supremum  of the sizes of all such is called the inverted largest simplex constant, and will be denoted by 
$\xi_{\ybul}(D;x)$. It will turn out that this constant does not depend on the choice of the infinitesimal flag taken, leading to the common value $\xi(D;x)$. 
As a result of our efforts we obtain a description  of moving Seshadri constants in all dimensions (cf. \cite{KL14}*{Theorem D}) 
in the following form.

\begin{theoremC}
Let $D$ be a big $\RR$-divisor on a smooth projective variety $X$, $x\notin\Bplus(D)$. Then 
\[
 \e (\|D\|;x) \equ \xi(D;x)\ .
\]
\end{theoremC}

Beside providing an alternative way of defining moving Seshadri constants, the largest inverted simplex constant has other benefits as well. Via Theorem~\ref{thm:bminus inf} and 
Theorem~\ref{thm:main augmented} it explains quite clearly why $\movs{D}{x}=0$  for a divisor $D$ with $x\in\Bplus(D)\setminus\Bminus(D)$. 

An interesting by-product of our result is a statement about the existence of global sections with prescribed vanishing behaviour. From the definition of Newton--Okounkov bodies 
it is a priori quite unclear which rational points arise as actual images of global sections, and in general it is very difficult to decide when it comes to boundary points. 
As it turns out, for infinitesimal Newton--Okounkov bodies the situation is more amenable.

\begin{corollaryD}(Corollary~\ref{cor:valuative})
Let $D$ be a big $\QQ$-divisor on $X$, $x\in X$ a closed point, and  $Y_{\bullet}$ an infinitesimal flag over  $x$. If 
$\Delta_{\xi}^{-1}\subseteq \Delta_{Y_{\bullet}}(\pi^*(D))$ for some $\xi>0$, then all vectors in $\Delta_{\xi}^{-1}\cap\QQ^n$
not lying on  the face generated by  the points $\lambda\cdot\eone,\lambda(\eone+\etwo), \ldots , \lambda(\eone+\en)$ are valuative. 
\end{corollaryD}

Finally, a somewhat tentative side remark regarding moving Seshadri constants and asymptotic multiplicities. For a given point $x\in X$, the loci of $\RR$-divisor classes 
in $\Bbig(X)$ where $\movs{D}{x}$ and $\asym{x}{D}$ are naturally defined are complementary, and we point out that one can glue these functions to a unique one via
\[
 \e_x(D) \deq \begin{cases} \movs{D}{x} & \text{ if } x\notin\Bplus(D) \\
			    0           & \text{ if } x\in\Bplus(D)\setminus\Bminus(D) \\
			    -\asym{x}{D} & \text{ if } x\in\Bminus(D)\ ,
              \end{cases}
\]
which ends up being homogeneous of degree one and  continuous on the big cone, while examples suggest that one can hope for $\e_x$ to be concave. We believe that $\e_x$ could prove useful as an extension of the moving Seshadri constant function by being capable of  distinguishing between divisor classes $D$ with $x\in\Bplus(D)\setminus\Bminus(D)$ and $x\in\Bminus(D)$. In the end we discuss an example where the Seshadri function is not everywhere differentiable on the ample cone.

A few words about the organization of the paper. We begin in Section 1 by fixing notation and collecting useful facts about asymptotic base loci, Newton--Okounkov bodies, 
and moving Seshadri constants, in Section 2 we present some important observations about infinitesimal Newton--Okounkov bodies. The characterization of restricted 
base loci is given in Section 3, while Section 4 is devoted to the main part of the paper, the description of augmented base loci in terms of Newton--Okounkov bodies 
with the help of separation of jets. Lastly, Section 5 hosts the discussion on Seshadri functions.

\smallskip 
\paragraph*{\bf Acknowledgements} We are grateful to Mihnea Popa for helpful discussions, and to the  Deutsche Bahn, the \"Osterreichische Bundesbahn, the SNCF and Thello
for providing us with  excellent working conditions. % Tomek Szemberg for EKL, 

\section{Notation and preliminaries}

\subsection{Notation}

We work over the complex number field, $X$ will stand for a projective variety of dimension $n$  which will often taken to be smooth. The point $x\in X$ will always be assumed a smooth point, while all points on varieties are taken to be closed. A divisor is always Cartier, whether it is integral, $\QQ$- , or $\RR$-Cartier and $D$ will denote a big divisor without exception. 

If $F$ is an effective $\RR$-Cartier divisor on $X$, then we write 
\[
 \mu_F(D) \equ \mu(D;F) \deq \sup\st{t>0\mid \text{$D-tF$ is big}}\ .
\]
Furthermore, if $Z\subseteq X$ is a smooth subvariety, then denote by
\[
 \mu_Z(D) \equ \mu(D;Z) \deq \mu(\pi^*D;E)\ ,
\]
where $\pi\colon X'\to X$ denotes the blow-up of $X$ along $Z$ with exceptional divisor $E$. 

\begin{remark}
Based on the definition of moving Seshadri constant given below, it is not hard to see that $0<\epsilon(\|D\|;x) \leq \mu(D;x)$. 
\end{remark}

\subsection{Asymptotic base loci}

Following \cite{ELMNP1}, one defines the restricted base locus of a big $\RR$-divisor $D$ as 
\[
\Bminus(D) \deq  \bigcup_{A} \Bstable(D+A)\ ,
\]
where the union is taken over all ample divisors $A$, such that $D+A$ is a $\QQ$-divisor. This locus is   a countable union of subvarieties of $X$  by   \cite{ELMNP1}*{Proposition~1.19}
\[
\Bminus(D) \equ \bigcup_{m\in\NN}\textbf{B}(D+\frac{1}{m}A)\ .
\]
The augmented base locus of $D$ is defined by 
\[
\Bplus(D)\deq  \bigcap_{A}\Bstable(D-A)\ ,
\]
where the intersection is taken again over all ample divisors $A$, such that $D+A$ is a $\QQ$-divisor. It follows quickly from \cite{ELMNP1}*{Proposition 1.5} that 
$\Bplus(D)=\Bstable (D-\frac{1}{m}A)$ for all $m\gg 0$ and any fixed ample class $A$.

\begin{proposition}\label{prop:openclosed}
Let $X$ be a  projective variety, $x\in X$ an arbitrary point. Then 
\begin{enumerate}
 \item $B_+(x) \deq \st{\alpha\in N^1(X)_\RR\mid x\in \Bplus(\alpha)} \dsubseteq N^1(X)_\RR$ is closed, 
 \item $B_-(x) \deq \st{\alpha\in N^1(X)_\RR\mid x\in \Bminus(\alpha)} \dsubseteq N^1(X)_\RR$ is open, 
\end{enumerate}
 both with respect to the metric topology of $N^1(X)_\RR$.  
\end{proposition}

For further references and relevant properties of restricted/augmented base loci, we refer the reader to \cites{ELMNP1,KL15}, including the proof to Proposition~\ref{prop:openclosed}.

\subsection{Newton--Okounkov bodies}

Newton--Okounkov bodies have been introduced to projective geometry by Lazarsfeld--Musta\c t\u a \cite{LM} and Kaveh--Khovanskii \cite{KKh} motivated by 
earlier work of Okounkov in representation theory \cite{Ok}.  For a big $\RR$-divisor $D$ on $X$, $\nob{\ybul}{D}$ stands for the Newton--Okounkov body of $D$ 
with respect to the admissible flag $\ybul$, where 
\[
 \ybul \ \colon \ X\equ \ Y_0 \ \supseteq \ Y_1 \ \supseteq \ \ldots \ \supseteq \ Y_n
\]
is a full flag of (irreducible) subvarieties $Y_i\subseteq X$ with $\codim_X Y_i=i$ and the property that $Y_i$ is smooth at the point $Y_n$ for all $0\leq i\leq n$. 
In particular, if $X$ is only assumed to be projective, the center $Y_n=\st{x}$ of an admissible flag must be a smooth point. 

\begin{remark}(Geometry of $\nob{\ybul}{D}$)
In low dimensions the geometry of $\nob{\ybul}{D}$ is well-understood: for curves $\nob{\ybul}{D}=[0,\deg D]\subseteq \RR$ is a line segment (\cite{LM}*{Example 1.13}); 
in the case of surfaces  variation of Zariski decomposition \cite{BKS} leads to the fact that Newton--Okounkov bodies are polygons with rational slopes 
(see \cite{LM}*{Theorem 6.4} and   \cite{KLM1}*{Section 2}). 

Note  that in dimensions three and above, the situation is no longer purely combinatorial: $\nob{\ybul}{D}$ can be non-polyhedral  even if $D$ is ample and $X$ is a Mori dream  
space. At the same time finite generation of the section ring of $D$ ensures the existence of flags with respect to which $\nob{\ybul}{D}$ is a rational simplex (see
\cite{AKL}).  
\end{remark}

Next,  we quickly recall a few notions and useful facts from \cite{KL15} without proof. 

\begin{proposition}[Equivalent definition of Newton-Okounkov bodies]\label{prop:definition}
Let $\xi\in\textup{N}^1(X)_{\RR}$ be a big $\RR$-class and $Y_{\bullet}$ be an admissible flag on $X$. Then
\[
\Delta_{Y_{\bullet}}(\xi) \ = \ \textup{closed convex hull of\ }\{ \nu_{Y_{\bullet}}(D) \ | \ D\in \textup{Div}_{\geq 0}(X)_{\RR}, D\equiv \xi\},
\]
where the valuation $\nu_{Y_{\bullet}}(D)$, for an effective $\RR$-divisor $D$, is constructed inductively as in the case of integral divisors.
\end{proposition}

\begin{proposition}\label{prop:compute}
Suppose $\xi$ is a big $\RR$-class and $Y_{\bullet}$ is an admissible flag on $X$. Then for any $t\in [0,\mu_{Y_1}(\xi))$, we have
\[
\Delta_{Y_{\bullet}}(\xi)_{\nu_1\geq t} \ = \ \Delta_{Y_{\bullet}}(\xi-tY_1)\ + t\eone,
\]
where $\eone=(1,0,\ldots ,0)\in \RR^n$.
\end{proposition}

\begin{lemma}\label{lem:nested}
Let $D$ be a big $\RR$-divisor and  $Y_\bullet$  an admissible flag on $X$.  Then the following hold.
\begin{enumerate}
 \item For any real number $\epsilon>0$ and any ample $\RR$-divisor $A$ on $X$, we have $\Delta_{Y_\bullet}(D) \subseteq \Delta_{Y_\bullet}(D+\epsilon A)$.
\item If $\alpha$ is an arbitrary nef  $\RR$-divisor class, then $\nob{\ybul}{D}\subseteq \nob{\ybul}{D+\alpha}$.
\item If $\alpha_m$ is any sequence of nef $\RR$-divisor classes with the property that $\alpha_m-\alpha_{m+1}$ is nef and 
$\|\alpha_m\|\to 0$ as $m\to\infty$  with respect to some norm on $N^1(X)_\RR$, then 
\[
 \Delta_{\ybul}(D) \equ \bigcap_m \nob{\ybul}{D+\alpha_m}\ .
\]
\end{enumerate}
\end{lemma}

\begin{definition}(Valuative points)
 Let $X$ be a projective variety, $\ybul$ an admissible flag, and $D$ a big $\QQ$-Cartier divisor on $X$. We call a point $v\in \nob{\ybul}{D}$ \emph{valuative},
 if it lies in the image of normalized map $\frac{1}{m}\nu_{\ybul}\colon |mD|\to \QQ_{\geq 0}$ for some $m\geq 1$, whenever $mD$ becomes Cartier.
\end{definition}

\begin{lemma}\label{lem:valuative}
With notation as above, $\intt\, \nob{\ybul}{D}\cap\QQ^n$ consists of valuative points. If $\nob{\ybul}{D}$ contains a small simplex with valuative 
vertices, then all  rational points of the simplex are valuative. 
\end{lemma}
\begin{proof}
Follows from Proposition~\ref{prop:definition} and multiplicative property of $\nu_{\ybul}$. 
\end{proof}

\subsection{Moving Seshadri constants}

We recall the  necessary information about moving Seshadri constants; our main source is \cite{ELMNP2}*{Section 6}. 

\begin{definition}(Moving Seshadri constant)
 Let $X$ be a projective variety, $x\in X$ be a smooth point, and $D$ a big $\RR$-divisor with $x\notin\Bplus(D)$. The \emph{moving Seshadri constant of $D$ at $x$} is defined as 
 \[
  \movs{D}{x} \deq \sup_{f^*D=A+E} \e(A;x)\ ,
 \]
where the supremum is taken over all projective morphisms $f\colon Y\to X$ with $Y$ smooth and $f$ an isomorphism around $x$, and over all decompositions $f^*D=A+E$, 
where $A$ is ample, and $E$ is effective with $f^{-1}(x)\notin\Supp (E)$. 
\end{definition}

If $D$ is nef, then $\movs{D}{x}$ specializes to the usual Seshadri constant $\e(D;x)$. The formal rules that the moving Seshadri constant obeys can be concisely expressed as follows.  

\begin{proposition}\cite{ELMNP2}*{Proposition 6.3}
With notation as above, $\movs{\ \cdot\ }{x}$ descends to a degree one homogeneous concave function on $\Bbig(X)\setminus B_+(x)$. 
\end{proposition}

By virtue of its concavity and the fact that its domain $\Bbig(X)\setminus B_+(x)\subseteq N^1(X)_\RR$ is open, $\movs{\ \cdot\ }{x}$ is of course a continuous function on it. The highly non-trivial result of \cite{ELMNP2} is that continuity is preserved under extending 
$\movs{\ \cdot\ }{x}$ by zero outside $\Bbig(X)\setminus B_+(x)$ in $\textup{N}^1(X)_\RR$. 

\begin{theorem}\cite{ELMNP2}*{Theorem 6.2}
Let $X$ be a smooth projective variety, $x\in X$. Then the function $\movs{\ \cdot\ }{x} \colon \textup{N}^1(X)_\RR  \to \RR_{\geq 0}$ given by 
\[ 
D   \mapsto  \begin{cases} \movs{D}{x} & \text{, if  $D\notin B_+(x)$} \\ 0 & \text{, otherwise} \end{cases}
\]
is continuous. 
\end{theorem}

In Section~5,  we offer an alternative extension of $\movs{D}{x}$ over $B_+(x)$.

\section{Infinitesimal Newton--Okounkov bodies}

In this section we define infinitesimal Newton--Okounkov bodies and  discuss some of their properties  needed in the rest of the paper. 
Recall that  we denote by $\pi\colon X'\to X$ the blow-up of $X$ at $x$ with 
exceptional divisor $E$. As $x$ is smooth, $X'$ is again a projective variety, and $E$ is an irreducible Cartier divisor on $X'$, which is smooth as a subvariety of 
$X'$. 

\begin{definition}
We say that $\ybul$ is an  \emph{infinitesimal flag over the point} $x$, if $Y_1=E$ and each $Y_i$ is a linear subspace in $E\simeq\PP^{n-1}$ of dimension $n-i$. 
We will often write $Y_n=\{z\}$. An \emph{infinitesimal flag over $X$} is an infinitesimal flag over $x\in X$ for some smooth point $x$. 

The symbol $\inob{\ybul}{D}$ stands for an infinitesimal Newton--Okounkov body of $D$, that is, 
\[
 \inob{\ybul}{D} \deq \nob{\ybul}{\pi^*D} \subseteq \RR^n_+\ ,
\]
where  $\ybul$ is an infinitesimal flag over $x$.
\end{definition}

\begin{remark}(Difference in terminology)
Note the deviation  in terminology from \cite{LM}*{Section 5.2}; what Lazarsfeld and Musta\c t\u a call an infinitesimal Newton--Okounkov body, is in our 
language (following \cite{KL14})  the generic infinitesimal Newton--Okounkov body. 
\end{remark}

\begin{remark}
Recently, interesting steps in the infinitesimal direction have  been taken  by Ro\'e \cite{Roe}.
\end{remark}

We start with an observation explaining the shapes of the 'right' kind of simplices that play the role of standard simplices in the infinitesimal theory.

\begin{lemma}(cf. \cite{KL15}*{Lemma 3.4})\label{lem:2inf}
Let $X$ be a projective variety, $x\in X$ a smooth point, and $A$ an ample Cartier divisor on $X$. Then there exists a natural number $m_0$ such that 
for  any infinitesimal flag $Y_{\bullet}$ over $x$ and  for every  $m\geq m_0$ there exist global sections  
$s_0',\ldots ,s_n'\in \HH{0}{X'}{\sO_{X'}(\pi^*(mA))}$ for which 
\[
\nu_{Y_{\bullet}}(s_0') \equ  \origin\ ,\ \nu_{Y_{\bullet}}(s_1')\equ \eone\ ,\ \text{ and }\ \nu_{Y_{\bullet}}(s_i') \equ \eone+\ei, 
\text{ for every $2\leq i\leq n$,}
\]
 where $\{ \eone,\ldots ,\en\}\subseteq \RR^n$ denotes the standard basis.
\end{lemma}

\begin{proof}
The line bundle $A$ is ample, therefore  there exists a natural number $m_0>0$ such that $m_0A$ is very ample, in particular the   linear series $|(m_0+m)A|$ define  
embeddings for all  $m\geq 0$.  As $|m_0A|$ separates tangent directions as well,  Bertini's theorem yields the  existence of  hyperplane sections 
$H_1,\ldots, H_{n-1} \in |m_0 A|$ intersecting  transversally at $x$, and $\tilde{H}_1\cap\ldots\cap\tilde{H}_i\cap E = Y_{i+1}$ for all $i=1,\ldots, n-1$, 
where $\tilde{H}_i$ denotes  the strict transform of $H_i$ through the blow-up map $\pi$.

At the same time observe that for any $m\geq m_0$ there exists a global section $t\in H^0(X,\sO_X(mA))$ not passing through $x$. By setting $s_i'\deq \pi^*(t+s_i)$ 
where $s_i\in\HH{0}{X}{\sO_X(m_0A)}$ is  a section associated to $H_i$, then the sections $s_0',\ldots,s_n'$ satisfy the requirements.
\end{proof} 

\begin{definition}
For a positive real number $\xi\geq 0$, the \emph{inverted standard simplex of size $\xi$}, denoted by $\Delta_{\xi}^{-1}$, is the convex hull of the set
\[
 \Delta_{\xi}^{-1} \ \deq \ \st{\origin, \xi \eone,\xi(\eone+\etwo),\ldots,\xi(\eone+\en)} \dsubseteq \RR^n .
\]
When $\xi =0$, then $\Delta_{\xi}^{-1}=\origin$.
\end{definition}

A major difference from  the non-infinitesimal case is the fact that infinitesimal Newton--Okounkov bodies are also contained in inverted simplices 
in a very natural way.

\begin{proposition}\label{prop:inverted}
Let $D$ be a big $\RR$-divsor $X$, then  $\Delta_{Y_{\bullet}}(\pi^*(D))\subseteq \iss{\mu(D;x)}$ 
for any  infinitesimal flag $Y_{\bullet}$ over the point $x$. 
\end{proposition}

\begin{proof}
By the continuity of Newton--Okounkov bodies inside the big cone it suffices  to treat the case when $D$ is a big $\QQ$-divisor. Homogeneity then lets 
us assume that $D$ is integral. Set $\mu=\mu(D;x)$. 

We will follow the line of thought of the proof of \cite{KL14}*{Proposition 3.2}. Recall that  $E\simeq \PP^{n-1}$;  we will write  $[y_1:\ldots :y_{n}]\in\PP^{n-1}$ 
for a set of homogeneous coordinates in $E$ such that 
\[
 Y_i \equ \Zeroes(y_1,\dots,y_{i-1}) \dsubseteq \PP^{n-1} \equ E\  \text{for all $2\leq i\leq n$.}
\]
With respect to a system of local coordinates  $(u_1,\ldots ,u_n)$ at  the point $x$, the blow-up $X'$ can be described (locally around $x$) as 
\[
X' \equ  \big\{\big((u_1,\ldots ,u_n);[y_1:\ldots :y_{n}]\big) \ | \ u_iy_j=u_jy_i \text{ for any } 1\leq i<j\leq n \big\}\ .
\] 
We can then write a global section $s$ of $D$  in the form 
\[
s \equ  P_m(u_1,\ldots ,u_n)+P_{m+1}(u_1,\ldots ,u_n)+\ldots + P_{m+k}(u_1,\ldots u_n) 
\]
around $x$,  where $P_{i}$ are homogeneous polynomials of degree $i$. 

We will perform the computation in the open subset $U_n=\{y_n\neq 0\}$, where we can take $y_n=1$ and the defining equations of the blow-up 
are given by $u_i=u_ny_i$ for  $1\leq i\leq n-1$. Then
\[
s|_{U_n} \ \equ \ u_n^m\cdot \big(P_m(y_1,\ldots ,y_{n-1},1)+u_nP_{m+1}(y_1,\ldots ,y_{n-1},1)+\ldots +u_n^kP_{m+k}(y_1,\ldots ,y_{n-1},1)\big)\ ,
\]
in particular,  $\nu_1(s)=m$. Notice that for the rest of $\nu_i(s)$'s  we have to restrict to the exceptional divisor $u_n=0$ and 
thus the only term arising  in the computation is  $P_m(y_1,\ldots ,y_{n-1},1)$. 

As $\deg P_m\leq m$, taking into account the algorithm for constructing the valuation vector of a section one can see that indeed 
\[
\nu_2(s)+\ldots +\nu_n(s) \ \leq \ \nu_1(s) \ ,
\]
and this finishes the proof of the proposition.
\end{proof}

\section{Restricted base locus via Newton--Okounkov bodies}
 
The section is devoted to  our characterization of restricted base loci in terms of infinitesimal data. The proofs are variations of those found in \cite{KL15}*{Section 2}.

\begin{theorem}\label{thm:bminus inf}
 Let $X$ be a smooth projective variety, $D$ a big $\RR$-divisor and $x\in X$ an arbitrary point on $X$. Then the following are equivalent. 
 \begin{enumerate}
  \item $x\not\in \Bminus(D)$.
  \item There exists an infinitesimal flag $\ybul$ over  $x$ such that $\origin\in\inob{\ybul}{D}$.  
   \item For every infinitesimal flag $\ybul$ over  $x$, one has $\origin\in \inob{\ybul}{D}$.
 \end{enumerate}
\end{theorem}

\begin{proof}
 $(1)\Rightarrow(3)$ Assume   $x\notin \Bminus(D)$, and fix a sequence of ample $\RR$-divisor $(\alpha_m)_{m\in\NN}$ so that $\alpha_m-\alpha_{m+1}$ is ample and $D+\alpha_m$ is a $\QQ$-divisor for any $m\geq 1$, and  $\|\alpha_m\|\to 0$ as $m\to \infty$. 
 
Now, let $\ybul$ be an arbitrary infinitesimal flag over  $x$. Since $x\notin\Bminus(D)$, then  $x\notin \Bstable(D+\alpha_m)$ for all $m\geq 1$. 
On the other hand, we have the sequence of equalities
 \[
  \Bstable(\pi^*(D+\alpha_m)) \equ \pi^{-1}( \Bstable(D+\alpha_m)) \ .
 \]
In particular, this implies that
\[
 E\cap \Bstable(\pi^*(D+\alpha_m)) \equ \varnothing \ ,
\]
for all $m\geq 1$. As $\ybul$ is an infinitesimal flag over $x$,  there must exist a sequence of natural numbers $n_m\geq 1$ and a  sequence of  global sections  $s_m\in H^0(X',\sO_{X'}(\pi^*(n_m(D+\alpha_m))))$ such that $s_m(z)\neq 0$. This implies that $\nu_{Y_{\bullet}}(s_m)=\origin$ for each $m\geq 1$. 
In particular, $\origin\in \inob{\ybul}{D+\alpha_m}$ for every $m\geq 1$.

Recall that $\pi^*\alpha_m$ is big and semi-ample, therefore 
\[
 \inob{\ybul}{D} \equ \bigcap_{m=1}^{\infty} \inob{\ybul}{D+\alpha_m}\ ,
\]
according to Lemma~\ref{lem:nested}, hence $\origin\in\inob{\ybul}{D}$ as wanted.  

The  implication $(3)\Rightarrow (2)$ is trivial, and so we are left with checking $(2)\Rightarrow (1)$. Let   $Y_{\bullet}$ be an infinitesimal flag over $x$ so that  $\origin\in\inob{\ybul}{D}$. Fix an ample $\RR$-divisor $A$ on $X$ and an decreasing sequence of positive real number $(t_m)_{m\in\NN}$  such that $\|t_m\|\to 0$ as $m\to\infty$, and  $D+t_mA$ is  a $\QQ$-divisor for all $m\geq 1$. Now, by Lemma~\ref{lem:nested}, we know
\[
 \origin \in \inob{\ybul}{D} \dsubseteq  \inob{\ybul}{D+t_mA} 
\]
for all $m\geq 0$, therefore $\min\sigma_{\pi^*(D+t_mA)}=0$ for the sum function  $\sigma_{\pi^*(D+t_mA)}\colon\inob{\ybul}{D+t_mA}\rightarrow\RR_{+}$. In particular, this implies, by making use of \cite{KL15}*{Proposition~2.6}, that $\mult_z(||\pi^*(D+t_mA)||)=0$ for all $m\geq 1$, where $Y_n=z$ is the base point of the flag $Y_{\bullet}$. Taking into account the string of (in)equalities 
\[
 \mult_x(\| D+t_mA\|) \equ \mult_E(\| \pi^*(D+t_mA)\|) \dleq \mult_z(\| \pi^*(D+t_mA) \|) \equ 0
\]
yields $\mult_x\| D+t_mA\|=0$ for all $m\geq 1$. As all the divisors $D+t_mA$ were taken to be $\QQ$-divisors, \cite{ELMNP1}*{Proposition 2.8} leads to $x\notin\Bminus(D+t_mA)$ for all $m\geq 1$. But, since 
\[
 \Bminus(D) \equ \bigcup_{m} \Bminus(D+\alpha_m) \equ \bigcup_m\Bstable(D+\alpha_m)
\]
by \cite{ELMNP1}*{Proposition~1.19}, we are done. 
\end{proof}
 
\begin{remark}
 We point out that the implication $(1)\Rightarrow (3)$ remains true under the weaker assumptions that $X$ is a projective variety and $x\in X$ a smooth point. 
 For the converse the answer is unclear since the proof of $(2)\Rightarrow (1)$ uses \cite{ELMNP1}*{Proposition~2.8}, which in turn is verified with the help of  multiplier ideals and 
 Nadel vanishing. 
\end{remark}

\begin{corollary}\label{cor:nefness}
Let $X$ be a smooth projective variety,  $D$ a big $\RR$-divisor on $X$. Then the following are equivalent.
\begin{enumerate}
 \item $D$ is nef.
 \item For every point $x\in X$ there exists an infinitesimal flag $\ybul$ over $x$ such that $\origin\in\inob{\ybul}{D}$.
 \item The origin $\origin\in\inob{\ybul}{D}$ for every infinitesimal flag over $X$. 
\end{enumerate} 
\end{corollary}

\section{Augmented base loci, infinitesimal Newton--Okounkov bodies, and jet separation}

In this section, which is the core of the paper, we extend  the characterization of augmented base loci in terms of infinitesimal Newton--Okounkov 
bodies suggested by \cite{KL14}*{Theorem 3.8} to all dimensions (cf. \cite{KL15}*{Theorem B} as well). Our statement  can be seen as  a  generalization of 
Seshadri's criterion for ampleness. The argument   will pass through a study of the  connection between infinitesimal Newton--Okounkov bodies and jet separation.

\subsection{The main theorem and the largest inverted simplex constant}
 
\begin{theorem}\label{thm:main augmented}
 Let $X$ be a smooth projective variety, $x\in X$ an arbitrary (closed) point,  $D$ a big $\RR$-divisor on $X$. Then the following are equivalent.
 \begin{enumerate}
  \item $x\notin\Bplus(D)$.
  \item For every infinitesimal flag $\ybul$ over $x$ there is $\xi>0$ such that $\iss{\xi}\subseteq \inob{\ybul}{D}$. 
  \item There exists an infinitesimal $\ybul$ over $x$ and $\xi>0$  such that $\iss{\xi}\subseteq \inob{\ybul}{D}$.
 \end{enumerate}
\end{theorem}

As an immediate consequence via the equivalence of ampleness and $\Bplus$ being empty, we obtain 

\begin{corollary}\label{cor:ampleness}
  Let $X$ be a smooth projective variety and  $D$ a big $\RR$-divisor on $X$. Then the following are equivalent.
 \begin{enumerate}
  \item $D$ is ample.
  \item For every point $x\in X$ and every infinitesimal flag $\ybul$ over $x$ there exists a real number $\xi>0$ for which $\iss{\xi}\subseteq \inob{\ybul}{D}$. 
  \item For every point $x\in X$ there exists an infinitesimal flag $\ybul$ over $x$ and a real number $\xi>0$ such that  $\iss{\xi}\subseteq \inob{\ybul}{D}$. 
 \end{enumerate}
 \end{corollary}

We will first give a proof of implication $(1)\Rightarrow (2)$ from Theorem~\ref{thm:main augmented}.

\begin{proposition}\label{prop:not in Bplus implies inf simplex}
 Let $X$ be a projective variety, $x\in X$ a smooth point, and $D$ a big $\RR$-Cartier divisor on $X$. If $x\notin \Bplus(D)$,  then there exists a real number $\xi>0$ such that $\iss{\xi}\subseteq \nob{\ybul}{D}$ for any infinitesimal flag $\ybul$ over $x$ 
\end{proposition}
  
\begin{proof}
This is a modification of the proof of \cite{KL15}*{Theorem B}  using Lemma~\ref{lem:2inf};  the basic strategy is the same. 

Let us first suppose  that  $D$ is  $\QQ$-Cartier. By assumption $x\notin \Bplus(D) = \Bstable(D-A)$ for some small ample $\QQ$-Cartier  divisor $A$. Note also that by 
$\Bstable(\pi^*(D-A))=\pi^{-1}(\Bstable(D-A))$ this gives 
\[
\Bstable(\pi^*(D-A))\ \bigcap \ E \equ \varnothing
\]
as well. Choose a positive integer  $m$ large and divisible enough  
such that $\pi^*(mA)$  becomes integral,  satisfies the conclusions of  Lemma~\ref{lem:2inf}, and  $\Bstable(\pi^*(D-A))=\textup{Bs}(\pi^*(m(D-A)))$ 
set-theoretically.  

Since $z\notin \textup{Bs}(\pi^*(m(D-A)))$,  there exists a section $s\in H^0(X',\sO_{X'}(\pi^*(mD-mA)))$ with $s(z)\neq 0$, i.e. 
$\nu_{Y_{\bullet}}(s)=\origin$. Furthermore,  Lemma~\ref{lem:2inf} provides   global  sections $s_0,\ldots, s_n\in H^0(X',\sO_{X'}(\pi^*(mA)))$ such that 
$\nu_{Y_{\bullet}}(s_0)\equ \origin$, $\nu_{Y_{\bullet}}(s_1)=\eone$ and $\nu_{Y_{\bullet}}(s_i)=\eone+\ei$ for all $2\leq i\leq n$. 

Multiplicativity of the valuation map $\nu_{\ybul}$ then gives 
\[
\nu_{Y_{\bullet}}(s\otimes s_0)\equ \origin, \nu_{Y_{\bullet}}(s\otimes s_1)\equ \eone \text{ and } \nu_{Y_{\bullet}}(s\otimes s_i)=\eone+\ei \ \ \text{for all $2\leq i\leq n$.}
\]
By the construction of Newton--Okounkov bodies, then $\iss{1/m}\subseteq \inob{\ybul}{D}$. 

Next, let $D$ be a big $\RR$-divisor for which $x\notin\Bplus(D)$, and let  $A$ be an ample $\RR$-divisor such that  $D-A$ is a $\QQ$-divisor,  and 
$\Bplus(D)=\Bplus(D-A)$. Then we have $x\notin\Bplus(D-A)$, therefore 
\[
\iss{\xi}\dsubseteq \inob{\ybul}{D-A} \dsubseteq \inob{\ybul}{D} 
\]
for some positive number $\xi$, according to the $\QQ$-Cartier case and  Lemma~\ref{lem:nested}. 
\end{proof}

Just as in the surface case, $x\notin\Bplus(D)$ implies that $\inob{\ybul}{D}$ will contain an inverted standard simplex of some size, hence it makes sense to ask how large these simplices  can become (cf.  \cite{KL14}*{Definition 4.5}).
  
 \begin{definition}(Largest inverted simplex constant)
  Let $X$ be a projective variety, $x\in X$ a smooth point on $X$, and $D$ a big $\RR$-divisor with $x\notin \Bminus(D)$. For an infinitesimal flag $\ybul$ over $x$ write 
  \[
   \xi_{\ybul}(D;x) \deq \sup\st{\xi\geq 0\mid \iss{\xi}\subseteq \inob{\ybul}{D}}\ .
  \]
  The \emph{largest inverted simplex constant} $\xi(D;x)$ of $D$ at $x$ is then defined as 
  \[
   \xi(D;x) \deq \sup_{\ybul} \xi_{\ybul}(D;x)\ ,
  \]
  where $\ybul$ runs through all infinitesimal flags over $x$. Moreover, if $x\in \Bminus(D)$, then let $\xi(D;x)=0$. 
 \end{definition}

\begin{remark}\label{rmk:xi continuous}
As  Newton--Okounkov bodies are homogeneous, so is  $\xi(\ \cdot\ ;x)$ as a function on $N^1(X)_\RR$. Although it is not a priori clear if $\xi(\ \cdot\ ;x)$ should be 
continuous, a bit of thought will convince that this is indeed the case over the domain where $x\notin \Bplus(D)$. 

First,  Corollary~\ref{cor:infinitesimal uniform}  below shows  that 
$\xi_{Y_{\bullet}}(D;x)$ is in fact independent of  $Y_{\bullet}$, therefore we can use one flag for all $\RR$-divisor classes. The natural inclusion
\[
\Delta_{Y_{\bullet}}(D) + \Delta_{Y_{\bullet}}(D')\dsubseteq \Delta_{Y_{\bullet}}(D+D')
\]
shows that  $\xi(\cdot;x)$ is in fact a concave  function on 
$\Bbig(X)\setminus B_+(x)$. This  latter is an open   subset of $\textup{N}^1(X)_\RR$,  therefore $\xi(\ \cdot\ ;x)$ is continuous on its domain. For further results regarding continuity, we advise the reader to look at Corollary~\ref{cor:xi continuous} and  Section 5. 
\end{remark}

\begin{proposition}\label{prop:uniformity for inf flags}
Let $X$ be a normal projective  variety, $x\in X$ a smooth point and $D$ a big $\RR$-Cartier divisor on $X$. Assume that 
$\iss{\xi}\subseteq \inob{\ybul}{D}$ for some infinitesimal flag $\ybul$ over $x$. Then 
$\iss{\xi}\subseteq \inob{\ybul'}{D}$ for all infinitesimal flags $\ybul'$ over $x$. 
\end{proposition}

\begin{remark}
Normality is used  in \cite{FKL}*{Theorem A}, a key ingredient of the proof. The cited result studies the question when the support of an effective 
$\RR$-divisor  is contained in certain augmented base loci  in terms of  the variation of the volume function. 
\end{remark}

\begin{proof}
The argument below works only for $\QQ$-divisors, passing to the limit delivers the general case (recall that restricted Newton--Okounkov bodies behave in a continuous
fashion by \cite{LM}*{Example 4.22}).  Assume  that $D$ is a big $\QQ$-divisor on $X$
For $\xi'\in(0,\xi)$, write  $\Delta_{\xi'}^{n-1}\subseteq \RR^{n-1}$ for standard  simplex of size  $\xi'$ and dimension $n-1$. 

Our  goal is then  to show that 
\[
\Delta_{\ybul'}(\pi^*(D)-\xi'E) \bigcap \{0\}\times\RR^{n-1} \equ  \Delta_{\xi'}^{n-1}
\]
for any infinitesimal flag $\ybul'$ over $x$. By continuity it suffices to check this for rational values of $\xi'$. 

So,  fix a rational number $\xi'\in(0,\xi)$ and denote by  $B\deq \pi^*(D)-\xi'E$. Obviously, 
\[
\Delta_{\ybul}(B+\lambda E) \equ  \Delta_{\ybul}(\pi^*D-(\xi'-\lambda)E) 
\]
for any $\lambda<\xi'$. The condition $\iss{\xi}\subseteq \inob{\ybul}{D}$ and  Proposition~\ref{prop:compute} imply
\[
\vol_{\RR^n}\big(\Delta_{\ybul}(B+\lambda E) \big) \ > \ \vol_{\RR^n}\big(\Delta_{\ybul}(B)\big) 
\]
for any rational number $0<\lambda<\xi'$. Then \cite{LM}*{Theorem A} gives  $\vol_X(B+\lambda E)>\vol_X(B)$, which, via 
\cite{FKL}*{Theorem B} leads to  $E\nsubseteq \Bplus(B)$.

The significance of this condition is that it grants us access to the slicing theorem  \cite{LM}*{Theorem 4.24}. In particular, 
\[
\nob{\ybul|E}{B} \equ \nob{\ybul}{B}|_{x_1=0} \equ \Delta_{\xi'}^{n-1}\ , 
\]
where left-hand side denotes  the appropriate  restricted Newton--Okounkov body (see \cite{LM}*{$(2.7)$}). 
By the same token, since  $E\nsubseteq \Bplus(B)$,  we have 
\[
\vol_{X'|E}(B) \equ  (n-1)!\vol_{\RR^{n-1}}(\Delta_{Y_{\bullet}|E}(B)) \equ (n-1)!\vol_{\RR^{n-1}}(\Delta_{\xi'}^{n-1}) \ . 
\]
Note that both extremes are independet of the choice of the flag, hence we have 
\[
\vol_{\RR^{n-1}}(\Delta_{\ybul'|E}(B)) \equ  \vol_{\RR^{n-1}}(\Delta_{\xi'}^{n-1}) 
\]
for  any   infinitesimal flag $\ybul'$ on $X'$.  

It follows from  Proposition~\ref{prop:inverted} that 
\[
\Delta_{\ybul'|E}(B) \equ  \Delta_{\ybul'}(B)|_{x_1=0} \dsubseteq  \Delta_{\xi'}^{n-1}\ ,
\]
however, as the the two convex bodies have equal volume, they must coincide. This means that $\Delta_{\ybul'}(B)|_{x_1=0}=\Delta_{\xi'}^{n-1}$
as required. 
\end{proof}
  
\begin{corollary}\label{cor:infinitesimal uniform}
With notation as above, $\xi(D;x)=\xi_{\ybul}(D;x)$ for all infinitesimal flags $\ybul$ over $x$. 
\end{corollary}

\subsection{Inverted standard simplices and jet separation}

Arguably one of the most important ingredients of the proof of Theorem~\ref{thm:main augmented} is the following connection between 
infinitesimal Newton--Okounkov bodies and jet separation of adjoint bundles.

\begin{proposition}\label{prop:jet separation}(Infinitesimal Newton--Okounkov bodies and jet separation)
Let $X$ be an $n$-dimensional smooth projective variety, $D$ a big Cartier divisor,  and $x$ be a (closed) point on $X$. Assume that there exists  a positive real number $\epsilon$  and a natural number $k$ with the property that $\Delta^{-1}_{n+k+\epsilon}\subseteq \Delta_{\ybul'}(\pi^*(D))$ for every  infinitesimal flag $\ybul'$ over $x$. Then $K_X+D$ separates $k$-jets.
\end{proposition}
\begin{proof}
By definition (see \cite{D}, also \cite{PAGI}*{Definiton 5.1.15} and  \cite{PAGI}*{Proposition 5.1.19}), what we need to prove is that the 
restriction map 
\[
H^0(X,\sO_X(K_X+D)) \ \longrightarrow \ H^0(X,\sO_X(K_X+D)\otimes {\sO_{X,x}}/{\m^{k+1}_{X,x}}) 
\]
is surjective.

Transferring the question to the blow-up $X'$, this is equivalent to requiring 
\begin{equation}\label{eq:surjective}
H^0(X,\sO_{X'}(\pi^*(K_X+D))) \ \rightarrow \ H^0(X',\sO_{X'}(\pi^*(K_X+D))\otimes {\sO_{X'}}/{\sO_{X'}(-(k+1)E)})
\end{equation}
to be  surjective. 

In order to do check  surjectivity  in $(\ref{eq:surjective})$, let us write  $B\deq\pi^*(D)-(n+k)E$. By Proposition~\ref{prop:compute}, we have
\[
\Delta_{Y'_{\bullet}}(B) \equ  \Delta_{Y'_{\bullet}}(\pi^*(D))_{x_0\geq n+k} \ - \ (n+k,0,\ldots,0)  
\]
for any  infinitesimal flag $Y'_{\bullet}$ over the point $x$. In particular, $B$ is a big line bundle with the property that the origin $\origin\in\Delta_{Y'_{\bullet}}(B)$ for any infinitesimal flag $Y'_{\bullet}$. As a consequence of Theorem~\ref{thm:bminus inf}, we obtain that $\Bminus(B)\cap E=\varnothing$. Thus 
 $\Zeroes(\sJ(X',\|B\|))\cap E=\varnothing$ via \cite{ELMNP1}*{Corollary 10}. 

To finish off the proof, we will make use of a variant of the classical argument to deduce the required surjectivity. Recall  that $B=\pi^*D-(n+k)E$, and  $K_{X'}=\pi^*K_X+(n-1)E$, therefore we have the short exact sequence
\[
0\rightarrow \sO_{X'}(K_{X'}+B)\otimes \sJ(X',||B||)\rightarrow \sO_{X'}(\pi^*(K_X+D)) \rightarrow 
\sO_{X'}(\pi^*(K_X+D)) \otimes \big( \sZ\oplus \sO_{(k+1)E}\big)\rightarrow 0 \ ,
\]
where $\sZ$ stands for the structure sheaf determined by the closed subscheme associated to the ideal $\sJ(X',||B||))$; note that this latter has support disjoint from $E$.

Since $B$ is a big line bundle, by Nadel's vanishing for asymptotic multiplier ideals  \cite{PAGII}*{Theorem~11.2.12.(ii)}  we have
\[
H^1(X', \sO_{X'}(K_{X'}+B)\otimes \sJ(X',||B||)) \equ 0 \ ,
\]
therefore the restriction map 
\[
\HH{0}{X'}{\sO_{X'}(\pi^*(K_X+D))} \longrightarrow  \HH{0}{X'}{\sO_{X'}(\pi^*(K_X+D)) \otimes \big( \sZ\oplus \sO_{(k+1)E}\big)}   
\]
is surjective, but then so is 
\[
 \HH{0}{X'}{\sO_{X'}(\pi^*(K_X+D))} \longrightarrow \HH{0}{X'}{\sO_{X'}(\pi^*(K_X+D))\otimes \sO_{(k+1)E}}\ ,
\]
as required. 
\end{proof}  

Now we are in a position to finish the proof of Theorem~\ref{thm:main augmented}; our main tool is going to be the connection between moving Seshadri constants 
and largest inverted simplex constants via jet separation (cf. \cite{ELMNP2}*{Proposition 6.6})
 
\begin{proposition}\label{prop:seshadri inverted}
Let  $D$ be a big $\RR$-divisor on a smooth projective variety $X$ and $x\in X$ a closed point. If $\xi(D;x)>0$, then $\xi(D;x)=\epsilon(||D||;x)$. 
\end{proposition}

\begin{proof}
Let us first assume that  $D$ is a big $\QQ$-divisor; we wish to show that 
\begin{equation}\label{eq:jetsinverse}
\xi(D;x) \equ  \limsup_{m\rightarrow \infty} \frac{s(mD;x)}{m} \equ  \epsilon(||D||;x) \ , 
\end{equation}
where the latter equality is  \cite{ELMNP2}*{Proposition 6.6}. Then one can go on and use $(\ref{eq:jetsinverse})$ and Proposition~\ref{prop:jet separation} 
to deduce $\xi(D;x)=\movs{D}{x}$. 

Our first goal is to check  $\epsilon(||D||;x)\geq \xi(D;x)$. Since both expressions  are homogeneous, it will suffice  to show  $\epsilon(||D||;x)\geq n$
whenever $\xi(D;x)>n$. Let  $r>0$ be a natural number so that $rD$ becomes integral. Then, by homogeneity, $\xi(mrD;x)>mrn$, and  
Proposition~\ref{prop:jet separation} gives 
\[
s(K_X+mrD;x) \dgeq  mrn-n\ .
\] 
Consequently, by taking multiples we obtain 
\[
\frac{s(k(K_X+mrD);x)}{k} \ \geq \ mrn-n, \text{ for any } m,k\geq 1 \ ,
\]
in particular, by \cite{ELMNP2}*{Proposition 6.6} one has
\[
\epsilon(||K_X+mrD||;x) \equ  \limsup_{k\rightarrow\infty}\frac{s(k(K_X+mrD);x)}{k} \dgeq  mrn-n \ .
\]
On the other hand,  \cite{ELMNP2}*{Theorem 6.2} says that the function $N^1(X)_{\RR}\ni\alpha\mapsto \epsilon(||\alpha||;x)\in \RR_+$ is continuous, 
therefore 
\[
\epsilon((||D||;x) \equ  \frac{1}{r}\limsup_{m\rightarrow\infty}\frac{\epsilon(||K_X+mrD||;x)}{m} \ \geq \ n \ .
\]

For the converse  inequality $\epsilon(||D||;x)\leq \xi(D;x)$, we will show that whenever  $D$ is an integral divisor separating  $s$-jets at the point $x$
, then $\Delta_{s}^{-1}\subseteq\Delta_{Y_{\bullet}}(\pi^*(D))$ for any infinitesimal flag $Y_{\bullet}$ over $x$. 
Note that  Proposition~\ref{prop:uniformity for inf flags} shows that it suffices check this  for one such flag. 

To this end, choose a system of local coordinates $\{u_1,\ldots ,u_n\}$ at $x$ and choose the infinitesimal flag $\ybul$ in such a way  that each $Y_{i+1}$ is  given by the intersection of $E$ with the proper transforms of $u_1,\ldots ,u_i$. Because  $D$ separates $s$-jets at $x$, there exist sections $s_1,\ldots ,s_n\in H^0(X,\sO_X(D))$ such that $s_i=u_i^s$ locally. Analogously to the  proof of Lemma~\ref{lem:2inf}, we see  that $\nu_{Y_{\bullet}}(\pi^*(s_1))=s\cdot\eone$ and $\nu_{Y_{\bullet}}(\pi^*(s_i))=s\cdot(\eone+\textup{\textbf{e}}_{i})$. The origin is  contained in \inob{\ybul}{D} since $\xi(D;x)>0$.

Lastly, it remains to deal with the case when  $D$ is a big $\RR$-divisor, which we will do by reduction to the rational case. Fix  a sequence of ample $\RR$-divisors $(\alpha_l)_{l\in\NN}$ for which  $\lim_{l\rightarrow\infty}||\alpha_l||=0$ for an arbitrary norm on $\textup{N}^1(X)_{\RR}$,  $D+\alpha_l$ is a $\QQ$-divisor, and $\alpha_{l+1}-\alpha_{l}$ is ample for any $l\geq 1$. 

Then  Lemma~\ref{lem:nested} yields 
\[
\Delta_{Y_{\bullet}}(\pi^*(D)) \equ  \bigcap_{l\in\NN}\Delta_{Y_{\bullet}}(\pi^*(D+\alpha_l)) 
\]
for any infinitesimal flag $Y_{\bullet}$ over $x$. As a consequence,  $\lim_{l\rightarrow\infty}\xi(D+A_l;x) =\xi(D;x)$; however,  since each class 
$D+A_l$ is a $\QQ$-divisor, we know that $\xi(D+A_l)=\epsilon(||D+A_l||;x)$ for any $l\in\NN$. Continuity of moving Seshadri constants 
\cite{ELMNP2}*{Theorem 6.2} then concludes  the proof.
\end{proof}

\begin{proof}[Proof of Theorem~\ref{thm:main augmented}]
The implication $(1)\Rightarrow (2)$ has been taken care of in Proposition~\ref{prop:not in Bplus implies inf simplex}, as  $(2)\Rightarrow (3)$
is formal, we are left with $(3)\Rightarrow (1)$. However, if there exists an infinitesimal flag $\ybul$ over $x$ with an inverted standard 
simplex contained in it, then $\xi(D;x)>0$, hence Proposition~\ref{prop:seshadri inverted} yields $\movs{D}{x}=\xi(D;x)>0$, which by definition 
means $x\notin\Bplus(D)$. 
\end{proof}  

We obtain a sequence of interesting corollaries. 

\begin{corollary}\label{cor:seshadri inverted}
Let $D$ be a big $\RR$-divisor on a smooth projective variety $X$. Then 
\[
\xi(D;x) \equ  \epsilon(||D||;x)
\]
for any (closed) point $x\in X$.
\end{corollary}
\begin{proof}
If $x\notin \Bplus(D)$, then this is immediate  from Theorem~\ref{thm:main augmented} and  Proposition~\ref{prop:seshadri inverted}. 
If $x\in \Bplus(D)\setminus\Bminus(D)$, then $\xi(D;x)=0$ by Proposition~\ref{prop:uniformity for inf flags} and $\epsilon(||D||;x)=0$ by definition. 
In the last case  $x\in \Bminus(D)$, both invariants are zero by definition.
\end{proof}

\begin{corollary}\label{cor:xi continuous}
For a smooth projective variety $X$ and a point $x\in X$, the function 
\begin{eqnarray*}
 \xi(\ \cdot\ ;x) \colon N^1(X)_\RR & \longrightarrow & \RR_{+} \\ D & \mapsto & \xi(D;x)
\end{eqnarray*}
is continuous. 
\end{corollary}
\begin{proof}
Follows easily from Corollary~\ref{cor:seshadri inverted} and \cite{ELMNP2}*{Theorem 6.2}.
\end{proof}

\begin{corollary}\label{cor:valuative}
Let $D$ be a big $\QQ$-divisor on $X$, $x\in X$ a closed point, and  $Y_{\bullet}$ an infinitesimal flag over  $x$. If 
$\Delta_{\xi}^{-1}\subseteq \Delta_{Y_{\bullet}}(\pi^*(D))$ for some $\xi>0$, then all vectors in $\Delta_{\xi}^{-1}\cap\QQ^n$
not lying on  the face generated by  the points $\lambda\cdot\eone,\lambda(\eone+\etwo), \ldots , \lambda(\eone+\en)$ are valuative. 
\end{corollary}
\begin{proof}
This is a consequence of Lemma~\ref{lem:valuative}, the inequality
\[
\lim_{k\rightarrow\infty}\frac{s(kD;x)}{k} \ \geq \ \lambda \ ,
\]
and the definition of jet separation constants. One can see by the proof of Proposition~\ref{prop:seshadri inverted}  all  vectors with rational coordinates sitting 
on one of the rays coming out of the origin in the inverted simplex come from a basis for some power of the maximal ideal of  $x$.
\end{proof}

\section{The extended Seshadri function}

In this section we briefly discuss an extension of moving Seshadri constants completing in some ways the picture considered in \cite{ELMNP2}. We also give an example where the Seshadri constant function inside the ample cone is not everywhere differentiable.

First, recall the notion of asymptotic multiplicity: for a point $x\in X$ on a smooth projective variety $X$, the \textit{asymptotic multiplicity}
of a big $\RR$-divisor $D$ is defined as 
\[
\asym{x}{D}\ \deq \ \inf_{D'} \st{\mult_x(D')} \ ,
\]
where the minimum is over all effective $\RR$-divisors with $D'\equiv D$ (see  \cite{ELMNP1} for the general theory).

Note that $\asym{x}{D}>0$ precisely when $x\in\Bminus(D)$ by \cite{ELMNP1}*{Proposition 2.9}; in contrast with the various  largest simplex constants and 
the geometric definition of the moving Seshadri constant, $\asym{x}{D}$ concerns the situation when the point $x\in\Bminus(D)$. 
Our goal is to see this invariant through the eyes of infinitesimal Newton--Okounkov bodies, and use this relation to connect asymptotic multiplicities to moving Seshadri constants. 

\begin{proposition}
Let $D$ be an $\RR$-divisor on $X$, $x\in \Bminus(D)$, and denote by  $r\deq\asym{x}{D}$. Then for any infinitesimal flags $Y_{\bullet}$ over the point $x$, the following hold
 \begin{enumerate}
\item $\inob{\ybul}{D}\subseteq r\cdot \eone\ + \ \RR^n_+$. In particular,  $E\subseteq \Bplus(\pi^*(D)-rE)$.
\item  $E\nsubseteq \Bminus(\pi^*(D)-rE)$. In particular $\inob{\ybul}{D}\cap \{r\}\times\RR^{n-1}\neq \varnothing$.
\item The intersection $\inob{\ybul}{D}\cap \{r\}\times\RR^{n-1}$ has  empty interior in $\RR^{n-1}$. 
\end{enumerate}
\end{proposition}

\begin{proof}
$(1)$ As $x\in\Bminus(D)$, the asymptotic multiplicity $r=\asym{x}{D}$ is strictly positive. By the definition of asymptotic multiplicity coupled with the fact that $\mult_x(D') = \ord_E(\pi^*(D')) = \nu_1(\pi^*(D'))$ 
for any effective $\RR$-divisor $D'\equiv D$ we obtain, $\Delta_{Y_{\bullet}}(\pi^*(D))\subseteq r\cdot\eone+\RR_+^n$.

Take an arbitrary point $z\in E$ and an infinitesimal flag $\ybul$ centered at $z\in X'$. Proposition~\ref{prop:compute} implies that 
$\origin\notin\nob{\ybul}{\pi^*D-tE}$  for $0\leq t<r$. Then $z\in\Bminus(\pi^*D-tE)\subseteq\Bplus(\pi^*D-tE)$ follows from \cite{KL15}*{Theorem~A} 
for all $0\leq t<r$. Using Proposition~\ref{prop:openclosed}.(i), then we know that $B_+(z)$ is closed in the big cone and in particular this yields that $z\in\Bplus(\pi^*D-rE)$ as well. 
 
$(2)$ Observe that $\pi^*D-rE$ is big (it has the same volume as $D$ has), therefore $\pi^*D-(r+t)E$ is big for all $0<t\ll 1$. By the definition of asymptotic multiplicity, $\asym{E}{\pi^*D-(r+t)E}=0$ for all $0<t\ll 1$, in particular $E\nsubseteq \Bminus(\pi^*D-(r+t)E)$. But then $z\notin \Bminus(\pi^*D-(r+t)E$ for all rational values  $0<t<\ll 1$ provided  $z\in E$ is very general. Now, making use of Proposition~\ref{prop:openclosed}.(ii) we know that $\textup{Big}(X')_{\RR}\setminus B_{-}(z)$ is closed. In particular, this yields that $z\notin\Bminus(\pi^*D-rE)$. 
 
$(3)$ Let us first point out that $E\nsubseteq \Bplus(\pi^*(D)-(r+t)E)$ for any $0<t\ll 1$. To see this, recall that by $(2)$ above,  $\inob{\ybul}{D}\cap \{r\}\times\RR^{n-1}\neq\varnothing$. 
Second, $\inob{\ybul}{D}$ is a full dimensional convex body, therefore
\[
\vol_{X'}(\pi^*(D)-rE) \ > \ \vol_{X'}(\pi^*(D)-(t+r)E) \ ,
\]
by  \cite{LM}*{Theorem~A}. But then  \cite{FKL}*{Theorem~A} gives $E\nsubseteq \Bplus(\pi^*(D)-(r+t)E)$ for any $0<t\ll 1$.

To finish the proof, suppose for a contradiction that  
\[
 \vol_{\RR^{n-1}}(\inob{\ybul}{D}\cap\{r\}\times\RR^{n-1}) \ >\ 0\ .
\] 
By  the slicing theorem \cite{LM}*{Theorem~4.24} and the fact that $E\nsubseteq \Bplus(\pi^*(D)-(r+t)E)$ for any $0<t\ll 1$, we obtain 
\[
\lim_{t\rightarrow0} \Big(\textup{vol}_{X'|E}\big(\pi^*(D)-(t+r)E\big)\Big) \ \geq \ \textup{vol}_{\RR^{n-1}}\Big(\inob{\ybul}{D}\cap\{r\}\times\RR^{n-1}\Big) \ > \ 0 \ .
\]
On the other hand,  \cite{ELMNP2}*{Theorem 5.7} forces the limit on the left-hand side to be zero,  since $E$ is an irreducible component of  $\Bplus(\pi^*(D)-rE)$ by $(1)$, a 
contradiction.
\end{proof}

\begin{lemma}\label{lem:curious}
Let $(D_k)_{k\in\NN}$ be a sequence of big $\RR$-divisors on a smooth projective variety $X$  converging to a big $\RR$-divisor $D$, 
let $x\in X$ be a  point. Then
\begin{enumerate}
\item  If $\epsilon(||D_k||;x)>0$ for all $k\in\NN$, and $\lim_{k\rightarrow\infty}\epsilon(||D_k||;x)=0$, then $x\in\Bplus(D)\setminus\Bminus(D)$.
\item If $\mult_x(||D_k||)>0$ for all $k\in\NN$, and $\lim_{k\rightarrow\infty}\asym{x}{D_k}=0$, then $x\in \Bplus(D)\setminus\Bminus(D)$.
\end{enumerate}
\end{lemma}

\begin{proof}
$(1)$ By Corollary~\ref{cor:seshadri inverted} it is legal to write  $\xi_k\deq \xi(||D_k||;x)=\movs{D_k}{x}$ for each $k\in\NN$. Fixing an 
infinitesimal flag $Y_{\bullet}$ over $x$, by definition we have  $\origin \in \Delta_{\xi_k}^{-1}\subseteq\inob{\ybul}{D_k}$. 
By continuity of   Newton--Okounkov bodies we obtain   $\textbf{0}\in\Delta_{Y_{\bullet}}(\pi^*(D))$, we can conclude by  Theorem~\ref{thm:bminus inf}, 
$x\notin\Bminus(D)$. 

On the other hand,  $x\in\Bplus(D)$ follows from the continuity of the moving Seshadri constant as a function on the N\'eron-Severi space. 

$(2)$ Since $\asym{x}{D_k}>0$,  \cite{ELMNP1}*{Theorem B} implies  that $x\in\Bminus(D_k)$ for all $k\in\NN$. But then $x\in \Bplus(D_k)$ for 
all $k\in\NN$ as well, whence $x\in\Bplus(D)$ according to \cite{KL15}*{Proposition 1.2}. For $x\notin\Bminus(D)$ note that asymptotic 
multiplicity is continuous on the big cone (see \cite{ELMNP1}*{Theorem~A}), therefore $\asym{x}{D}=0$, and consequently $x\notin\Bminus(D)$.
\end{proof}

By Corollary~\ref{cor:seshadri inverted}, and Lemma~\ref{lem:curious} one can glue  the functions  $\epsilon(||\ \cdot\ ||;x)$ and $-\asym{x}{\ \cdot\ }$ 
giving rise to a continuous extension of the moving Seshadri constant function which is nowhere zero on the open subset $B_-(x)\subseteq \textup{N}^1(X)_{\RR}$.

\begin{definition}[Extended Seshadri function]\label{defi:Seshadri function}
Let $X$ be a smooth projective variety, $x\in X$. We define the \emph{(extended) Seshadri function} $\e_x\colon\Bbig(X)\to \RR_{\geq  0}$ 
\emph{associated to the point $x\in X$} by 
\[
 \e_x(D) \deq \begin{cases} \movs{D}{x} & \text{ if } D\notin B_+(x) \\
			    0           & \text{ if } D\in B_+(x)\setminus B_-(x)) \\
			    -\asym{x}{D} & \text{ if } D\in B_-(x)\ . 
              \end{cases}
\]
\end{definition}

\begin{remark}
Since both the asymptotic multiplicity and the moving Seshadri constant are concave on  the domain where they are  non-zero, it is not unnatural to hope that 
the extended Seshadri function will retain this property. We shall see that this is indeed the case in the example below. 
\end{remark}

We end this section with an explicit computation of the extended Seshadri function; an interesting feature of the example is that $\e_x$ is not everywhere differentiable even inside the ample cone. 

\begin{example}[A non-differentiable Seshadri function]
Let $p\in \PP^2$ be a point and denote by $\pi_1: X\deq Bl_p(\PP^2)\rightarrow \PP^2$ the blow-up of $\PP^2$ at the point $p$ with  exceptional divisor $E$. 
We pick a point $x\in E$, and then pursue to compute the function $\e_x$ on the pseudo-effective cone $\overline{\Eff}(X) = \RR_+E+\RR_+(H-E)$, 
where $H$ is the pullback of the class of a line. 

The function $\e_x$ being homogeneous of degree one, it will suffice to determine the values of $\e_x$ as we traverse the line segment 
$[E,H-E]\subseteq N^1(X)_\RR \equ \RR^2$. To this end, set 
\[
 F_t \deq  tH + (1-2t)E\ \ \text{ for all $0\leq t\leq 1$.}
\]

Observe that  for  $t\in[0,\frac{1}{2})$  we have $x\in\Bminus(F_t)$, and 
\[
\e_x(F_t) \equ - \asym{x}{tH+(1-2t)E} \equ 2t-1\ .
\]

If $\frac12 \leq t \leq 1$, then $F_t$ is nef, hence $\e_x(F_t) \equ \movs{F_t}{x} \equ \e(F_t;x)$. The  Seshadri constants $\e(F_t;x)$  
are somewhat more complicated  to compute,  this will  take up the remaining part of our  example. 

Thus, let $\pi_2:X'\rightarrow X$ denote the blow-up  of $X$ at the point $x$. Write  $\pi=\pi_2\circ\pi_1$ for the composition of the  two blow-ups. 
On $X'$ we have precisely three negative curves: 
\begin{eqnarray*}
 E_1 & = & \text{the strict transform of the exceptional divisor of $\pi_1$ under $\pi_2$}, \\
 E_2 & = & \text{the exceptional divisor of the blow-up $\pi_2$},\\
 E_3 & = & \text{the strict transform of the line of class $H-E$ on $X$ going through the point $x$}. 
\end{eqnarray*}
The intersection matrix of the curves $E_i$ is 
\[
 (E_i\cdot E_j)_{1\leq i \leq 3,1\leq j\leq 3} \equ \left( \begin{array}{ccc}
                                                            -2 & 1 & 0 \\ 1 & -1 & 1 \\ 0& 1 & -1 
                                                           \end{array}
                                                           \right)\ .
\]
In the basis $(E_1,E_2,E_3)$ of $N^1(X')_\RR$, the hyperplane class $H'$ is given as 
\[
 H' \equ E_1+2E_2+E_3\ .
\]
Along with $H'$, the divisors $H'+E_3$ and $E_2+E_3$ turn out to be nef as well, 
and the three generate the nef cone of $X'$.  In this notation, 
\[
D_t \deq \pi_2^*F_t \equ \pi_2^*(tH+(1-2t)E) \equ tH' + (1-2t)(E_1+E_2) \ ,
\]
which can in turn be written in the form 
\[
 D_t \equ (1-t)H' + (2t-1)(E_2+E_3)\ \ \text{for all $1/2\leq t\leq 1$.}
\]
This means in particular that $D_t$ sits on the face of the nef cone generated by $H'$ and $E_2+E_3$ for all $1/2\leq t\leq 1$. 

As one can check, that  the ray $D_t-\epsilon E_1$  leave the nef cone through the 
 the face generated by the  divisors $H'$ and $H'+E_3$ whenever $t\in 
[\frac{1}{2},\frac{2}{3}]$,  and throught the face generated by the divisors $H'+E_3$ and $E_1+E_3$ for  $t\in[\frac{2}{3},1]$. 

As a result, $\e_x$ is going to be piecewise linear, and  it is not going to be differentiable  
 at $t=\frac{2}{3}$. The full computation goes as follows. 

For $t\in [\frac{1}{2},\frac{2}{3}]$, the ray $D_t-\epsilon E_1$ hits the boundary of the nef cone at $\epsilon=2t-1$, in the divisor
\[
D_t-(2t-1)E_1 \equ  (2-3t)H'+(2t-1)(H'+E_3) \ .
\]
In particular, $\epsilon_x(F_t)=2t-1$ on the interval $[\frac{1}{2},\frac{2}{3}]$. 

On the other hand, if $t\in[\frac{2}{3},1]$, the ray $D_t-\epsilon E_1$ reaches the boundary of the nef cone at $\epsilon=1-t$, in the divisor
\[
D_t-(1-t)E_1 \equ  (1-t)(H'+E_3) + (3t-2)(E_1+E_3) \ , 
\]
and we obtain  $\epsilon_x(t)=1-t$ on the interval $[\frac{2}{3},1]$. Putting all this together, the Seshadri function on the line segment 
$[E,H-E]$ is given by 
\begin{displaymath}
\epsilon_x(F_t) \equ   \left\{ \begin{array}{lll}
2t-1 & \textrm{if $t\in[0,\frac{1}{2}]$}\\
2t-1 & \textrm{if $t\in[\frac{1}{2},\frac{2}{3}]$}\\
1-t & \textrm{if $t\in [\frac{2}{3},1]$} 
\end{array} \right. \ .
\end{displaymath}
\end{example}

%%%%%%%%%%%%%%%%%%%%%%%%%%%%%%%%%%%%%%%%%%%%  End of Text %%%%%%%%%%%%%%%%%%%%%%%%%%%%%%%%

\raggedright

\end{document}